\documentclass[10pt]{amsart}
\usepackage[leqno]{amsmath}
\usepackage{amssymb,latexsym,soul,cite,amsthm,color,enumitem,graphicx,tikz,mathtools,microtype,accents}
\usepackage[colorlinks=true,urlcolor=flame,citecolor=flame,linkcolor=flame,linktocpage,pdfpagelabels,bookmarksnumbered,bookmarksopen]{hyperref}
\definecolor{flame}{rgb}{0.89, 0.35, 0.13}
\usepackage[english]{babel}
\usepackage[left=2.5cm,right=2.5cm,top=2.5cm,bottom=2.5cm]{geometry}

\numberwithin{equation}{section}

\newtheorem{theorem}{Theorem}[section]
\theoremstyle{plain}
\newtheorem{lemma}[theorem]{Lemma}
\theoremstyle{plain}
\newtheorem{proposition}[theorem]{Proposition}
\theoremstyle{plain}

\theoremstyle{definition}
\newtheorem{remark}[theorem]{Remark}
\newtheorem{example}[theorem]{Example}

\newcommand{\N}{{\mathbb N}}

\newcommand{\R}{{\mathbb R}}
\newcommand{\eps}{\varepsilon}
\newcommand{\beq}{\begin{equation}}
\newcommand{\eeq}{\end{equation}}
\renewcommand{\le}{\leqslant}
\renewcommand{\ge}{\geqslant}

\newcommand{\restr}[2]{\left.#1\right|_{#2}}
\newcommand{\w}{W^{s,p}_0(\Omega)}
\newcommand{\fpl}{(-\Delta)_p^s\,}
\newcommand{\ds}{{\rm d}_\Omega^s}
\newcommand{\p}{p^*_s}
\newcommand{\cs}{C_s^0(\overline{\Omega})}

\makeatletter
\newcommand{\leqnomode}{\tagsleft@true}
\newcommand{\reqnomode}{\tagsleft@false}
\makeatother

\newenvironment{enumroman}{\begin{enumerate}

}{\end{enumerate}}

\title[Fractional $p$-Laplacian with asymmetric reactions]{Four solutions for fractional $p$-Laplacian equations with asymmetric reactions}

\author[A.\ Iannizzotto, R.\ Livrea]{Antonio Iannizzotto, Roberto Livrea}

\address[A.\ Iannizzotto]{Department of Mathematics and Computer Science
\newline\indent
University of Cagliari
\newline\indent
Via Ospedale 72, 09124 Cagliari, Italy}
\email{antonio.iannizzotto@unica.it}

\address[R.\ Livrea]{Department of Mathematics and Computer Science
\newline\indent
University of Palermo
\newline\indent
Via Archirafi 34, 90123 Palermo, Italy}
\email{roberto.livrea@unipa.it}

\begin{document}

\begin{abstract}
We consider a Dirichlet type problem for a nonlinear, nonlocal equation driven by the degenerate fractional $p$-Laplacian, whose reaction combines a sublinear term depending on a positive parameter and an asymmetric perturbation (superlinear at positive infinity, at most linear at negative infinity). By means of critical point theory and Morse theory, we prove that, for small enough values of the parameter, such problem admits at least four nontrivial solutions: two positive, one negative, and one nodal. As a tool, we prove a Brezis-Oswald type comparison result.
\vskip2pt
\noindent
{\em $2010$ Mathematics Subject Classification.} 35A15, 35R11, 58E05.
\end{abstract}

\maketitle

\begin{center}
Version of \today\
\end{center}

\section{Introduction}\label{sec1}

\noindent
Nonlinear elliptic partial differential equations with asymmetric nonlinearities are usually written in the form
\[Lu = f(x,u) \ \text{in $\Omega$},\]
with several boundary conditions, where $L$ is some elliptic operator, and $f:\Omega\times\R\to\R$ is a nonlinear reaction with qualitatively different behaviors as the second variable tends to $\pm\infty$, respectively. Typically, such asymmetric behavior can be exploited to prove, via variational or topological methods, the existence of multiple solutions to the equation.
\vskip2pt
\noindent
The study of such asymmetric problems, to our knowledge, dates back to the work of Motreanu, Motreanu and Papageorgiou \cite{MMP2,MMP1}, and was then developed by several authors considering a wide range of semilinear or quasilinear equations with Dirichlet, Neumann, or even Robin boundary conditions. We recall the results of \cite{CLP,IMP,MP,PR1,RR}.
\vskip2pt
\noindent
The present paper is devoted to the study of the following Dirichlet type problem:
\beq\label{dir}
\begin{cases}
\fpl u = \lambda|u|^{q-2}u+g(x,u) & \text{in $\Omega$} \\
u = 0 & \text{in $\Omega^c$.}
\end{cases}
\eeq
Here $\Omega\subset\R^N$ ($N\ge 2)$ is a bounded domain with $C^{1,1}$ boundary $\partial\Omega$, $s\in(0,1)$, $p\ge 2$ are s.t.\ $ps<N$, and the leading operator is the degenerate fractional $p$-Laplacian, defined for all $u:\R^N\to\R$ smooth enough and $x\in\R^N$ by
\beq\label{fpl}
\fpl u(x)=2\lim_{\eps\to 0^+}\int_{B_\eps^c(x)}\frac{|u(x)-u(y)|^{p-2}(u(x)-u(y))}{|x-y|^{N+ps}}\,dy
\eeq
(which for $p=2$ reduces to the linear fractional Laplacian up to a dimensional constant $C(N,s)>0$, see \cite{DPV}). The reaction in \eqref{dir} is the sum of two terms. The first, depending on a real parameter $\lambda>0$, is a $(p-1)$-sublinear power of the unknown, i.e., $q\in(1,p)$. The second is a Carath\'eodory mapping $g:\Omega\times\R\to\R$ subject to a global subcritical growth condition and combining a $(p-1)$-linear or superlinear behavior near $0$ with an asymmetric behavior at $\pm\infty$, namely, $g(x,t)$ is $(p-1)$-superlinear at $\infty$ and at most $(p-1)$-linear at $-\infty$.
\vskip2pt
\noindent
Elliptic equations driven by linear nonlocal operators (whose prototype is the fractional Laplacian) were first studied via variational methods in \cite{SV,SV1}, while regularity theory has its ground in \cite{RS}, giving rise to a wide literature (we refer the reader to the monograph \cite{MRS}). In the quasilinear case $p\neq 2$, things are obviously more involved. The eigenvalue problem for $\fpl$ was first studied in \cite{LL}, variational methods for equations with several types of reactions were established in \cite{ILPS}, H\"older regularity of weak solutions was studied in \cite{IMS1,IMS2} (for $p>2$), maximum and comparison principles were proved in \cite{DQ,J}, equivalence between Sobolev and H\"older minimizers of the energy functional was proved in \cite{IMS3}, and a detailed study of sub- and supersolutions was performed in \cite{FI}. Existence results for the fractional $p$-Laplacian with asymmetric reactions were obtained in \cite{HY,PMZ}, while closely related problems were studied in \cite{BM,BP,BS,CMS,DQ1,XZR}. For a more detailed discussion, we refer to the surveys \cite{MS,P}.
\vskip2pt
\noindent
Our approach to problem \eqref{dir} is variational, inspired by \cite{MMP1}. We encode weak solutions as critical points of a $C^1$ energy functional $\Phi_\lambda$, defined on a convenient fractional Sobolev space and depending on $\lambda>0$. Due to the presence of the asymmetric perturbation, $\Phi_\lambda$ has no definite asymptotic behavior, so we define two truncated functionals $\Phi^\pm_\lambda$ whose critical points coincide with the positive and negative solutions of \eqref{dir}, respectively. We prove that, for all $\lambda>0$ small enough, $\Phi^+_\lambda$ has at least two nonzero critical points, one given by the mountain pass theorem and a local minimizer. Besides, for all $\lambda>0$, $\Phi^-_\lambda$ contributes at least one global minimizer. So we have three nontrivial constant sign solutions (Theorem \ref{css}).
\vskip2pt
\noindent
Pushing forward our analysis, we see that, under slightly more restrictive hypotheses, for even smaller values of $\lambda>0$, problem \eqref{dir} admits a smallest positive solution and a biggest negative solution (an idea that was first introduced in \cite{DD}). So, we truncate again the reaction introducing a new energy functional $\tilde\Phi_\lambda$, which turns out to have one more critical point of mountain pass type (in the sense of Hofer \cite{H}), taking values between the extremal constant sign solutions. Finally, by a Morse theoretic argument we show that such critical point is not $0$, hence it turns out to be a nodal (sign-changing) solution of \eqref{dir}. Thus, we conclude that \eqref{dir} admits at least four nontrivial solutions for all $\lambda>0$ small enough (Theorem \ref{nod}).
\vskip2pt
\noindent
In proving the existence of the smallest positive solution, we do not apply (as usual in such cases, see \cite{MMP1}) the strong comparison principle of \cite{J}, since it requires rather restrictive assumptions on the data $p$, $s$. Instead, we present a special comparison result for sub-supersolutions under a monotonicity condition, inspired by the classical Brezis-Oswald work \cite{BO} (see \cite{DS,LPPS,MPV} for other versions). We believe that such comparison result (stated in Theorem \ref{com} below) can be useful also in different frameworks.
\vskip2pt
\noindent
Our result represents an application of classical methods in nonlinear analysis combined with the recently established theory for the fractional $p$-Laplacian (mainly the results of \cite{DQ,FI,IMS3}). To our knowledge, this is the first multiplicity result for a fractional order problem with asymmetric reaction, even in the linear case $p=2$.
\vskip2pt
\noindent
The paper has the following structure: in Section \ref{sec2} we collect some preliminary results on fractional $p$-Laplace equations and prove a comparison result; in Section \ref{sec3} we prove the existence of two positive and a negative solutions; and in Section \ref{sec4} we prove the existence of extremal constant sign solutions and of a nodal solution.
\vskip4pt
\noindent
{\bf Notation:} For any $A\subset\R^N$ we shall set $A^c=\R^N\setminus A$. For any two measurable functions $u,v:\Omega\to\R$, $u\le v$ will mean that $u(x)\le v(x)$ for a.e.\ $x\in\Omega$ (and similar expressions). The positive (resp., negative) part of $u$ is denoted $u^+$ (resp., $u^-$). Every function $u$ defined in $\Omega$ will be identified with its $0$-extension to $\R^N$. If $X$ is an ordered Banach space, then $X_+$ will denote its non-negative order cone. The open and closed balls, respectively, centered at $u$ with radius $\rho>0$ will be denoted $B_\rho(u)$, $\overline{B}_\rho(u)$. For all $r\in[1,\infty]$, $\|\cdot\|_r$ denotes the standard norm of $L^r(\Omega)$ (or $L^r(\R^N)$, which will be clear from the context). Moreover, $C$ will denote a positive constant (whose value may change case by case).

\section{Preliminaries}\label{sec2}

\noindent
In this section, for the reader's convenience, we recall some basic results about the general Dirichlet problem for the degenerate fractional $p$-Laplacian (some also hold in the singular case $p\in(1,2)$):
\beq\label{p0}
\begin{cases}
\fpl u = f(x,u) & \text{in $\Omega$} \\
u = 0 & \text{in $\Omega^c$,}
\end{cases}
\eeq
where $\Omega$, $p$, $s$ are as in the Introduction and $f$ satisfies the following hypotheses:
\begin{itemize}[leftmargin=1cm]
\item[${\bf H}_0$] $f:\Omega\times\R\to\R$ is a Carath\'eodory function, and there exist $c_0>0$, $r\in(p,\p)$ s.t.\ for a.e.\ $x\in\Omega$ and all $t\in\R$
\[|f(x,t)| \le c_0(1+|t|^{r-1}).\]
\end{itemize}
By $\p$ we denote the critical fractional Sobolev exponent, namely, $\p=Np/(N-ps)$. Also, for all $(x,t)\in\Omega\times\R$ we set
\[F(x,t) = \int_0^t f(x,\tau)\,d\tau.\]
We provide problem \eqref{p0} with a variational structure, following \cite{FI}. For all measurable $u:\R^N\to\R$ define the Gagliardo seminorm
\[[u]_{s,p} = \Big[\iint_{\R^N\times\R^N}\frac{|u(x)-u(y)|^p}{|x-y|^{N+ps}}\,dx\,dy\Big]^\frac{1}{p}.\]
We define the fractional Sobolev spaces
\[W^{s,p}(\R^N) = \big\{u\in L^p(\R^N):\,[u]_{s,p}<\infty\big\},\]
\[\w = \big\{u\in W^{s,p}(\R^N):\,u=0 \ \text{in $\Omega^c$}\big\},\]
the latter being a uniformly convex, separable Banach space under the norm $\|u\|=[u]_{s,p}$, with dual space $W^{-s,p'}(\Omega)$ (see \cite{DPV}). The embedding $\w\hookrightarrow L^q(\Omega)$ is continuous for all $q\in[1,p^*_s]$ and compact for all $q\in[1,p^*_s)$. For any $u\in\w$ we can define $\fpl u \in W^{-s,p'}(\Omega)$ by setting for all $v\in\w$
\[\langle \fpl u, v\rangle = \iint_{\R^N\times\R^N} \frac{|u(x)-u(y)|^{p-2}(u(x)-u(y))(v(x)-v(y))}{|x-y|^{N+ps}}\,dx\,dy.\]
The definition above agrees with \eqref{fpl} when $u\in\mathcal{S}(\R^N)$. By \cite[Lemma 2.1]{FI}, $\fpl:\w\to W^{-s,p'}(\Omega)$ is a monotone, continuous, $(S)_+$-operator. Besides, the following inequality holds for all $u,v\in\w$ as an immediate consequence of H\"older's inequality:
\beq\label{dua}
\langle\fpl u,v\rangle \le \|u\|^{p-1}\|v\|.
\eeq
Since the mapping $t\mapsto t^+$ is Lipschitz, for all $u\in\w$ we have $u^\pm\in\w$, but in general
\[\|u\|^p \neq \|u^+\|^p+\|u^-\|^p,\]
unlike in the case of the classical Sobolev space $W^{1,p}_0(\Omega)$. The following lemma illustrates some simple properties of positive and negative parts, which will be used in our arguments:

\begin{lemma}\label{pn}
Let $u\in\w$, then:
\begin{enumroman}
\item\label{pn1} $\|u^\pm\|\le\|u\|$;
\item\label{pn2} $\|u^\pm\|^p\le\langle\fpl u,\pm u^\pm\rangle$.
\end{enumroman}
\end{lemma}
\begin{proof}
We only deal with $u^+$ (the argument for $u^-$ is analogous). Set
\[A_+=\big\{x\in\R^N:\,u(x)>0\big\}, \ A_-=A_+^c.\]
Then we have
\begin{align*}
\|u^+\|^p &= \iint_{\R^N\times\R^N}\frac{|u^+(x)-u^+(y)|^p}{|x-y|^{N+ps}}\,dx\,dy \\
&= \iint_{A_+\times A_+}\frac{|u(x)-u(y)|^p}{|x-y|^{N+ps}}\,dx\,dy+\iint_{A_+\times A_-}\frac{u(x)^p}{|x-y|^{N+ps}}\,dx\,dy+\iint_{A_-\times A_+}\frac{u(y)^p}{|x-y|^{N+ps}}\,dx\,dy \\
&\le \iint_{A_+\times A_+}\frac{|u(x)-u(y)|^p}{|x-y|^{N+ps}}\,dx\,dy+\iint_{A_+\times A_-}\frac{(u(x)-u(y))^p}{|x-y|^{N+ps}}\,dx\,dy+\iint_{A_-\times A_+}\frac{(u(y)-u(x))^p}{|x-y|^{N+ps}}\,dx\,dy \\
&\le \iint_{\R^N\times\R^N}\frac{|u(x)-u(y)|^p}{|x-y|^{N+ps}}\,dx\,dy \\
&= \|u\|^p,
\end{align*}
which proves \ref{pn1}. Besides, by \cite[Lemma A.2]{BP} (with $g(t)=G(t)=t^+$) we have for all $a,b\in\R$
\[|a-b|^{p-2}(a-b)(a^+-b^+) \ge |a^+-b^+|^p.\]
So we have
\begin{align*}
\langle\fpl u,u^+\rangle &= \iint_{\R^N\times\R^N}\frac{|u(x)-u(y)|^{p-2}(u(x)-u(y))(u^+(x)-u^+(y))}{|x-y|^{N+ps}}\,dx\,dy \\
&\ge \iint_{\R^N\times\R^N}\frac{|u^+(x)-u^+(y)|^p}{|x-y|^{N+ps}}\,dx\,dy = \|u^+\|^p,
\end{align*}
which proves \ref{pn2}.
\end{proof}

\noindent
A function $u\in\w$ is a (weak) solution of problem \eqref{p0} if for all $\varphi\in\w$
\[\langle\fpl u,\varphi\rangle = \int_\Omega f(x,u)\varphi\,dx.\]
Similarly, we say that $u$ is a (weak) supersolution of \eqref{p0} if for all $\varphi\in\w_+$
\[\langle\fpl u,\varphi\rangle \ge \int_\Omega f(x,u)\varphi\,dx.\]
The definition of a (weak) subsolution is analogous. For short, in such cases, we will say that $u$ satisfies weakly in $\Omega$
\[\fpl u = \ (\ge,\,\le) \ f(x,u).\]
If $u$ is a subsolution and $v$ is a supersolution s.t.\ $u\le v$ in $\Omega$, we say that $(u,v)$ is a sub-supersolution pair of \eqref{dir}, and we set
\[\mathcal{S}(u,v) = \big\{w\in\w: \ \text{$w$ is a solution of \eqref{p0}, $u\le w\le v$ in $\Omega$}\big\}.\]
The properties of the set $\mathcal{S}(u,v)$ are investigated in \cite[Lemmas 3.2 -- 3.4, Theorem 3.5]{FI} (even under a more general definition of sub- and supersolution):

\begin{proposition}\label{sss}
Let ${\bf H}_0$ hold, $(u,v)$ be a sub-supersolution pair of \eqref{p0}. Then, $\mathcal{S}(u,v)$ is a nonempty, compact set in $\w$, both upward and downward directed, in particular it has a smallest and a biggest element (with respect to the pointwise ordering of $\w$).
\end{proposition}

\noindent
As a special case of \cite[Theorem 3.3]{CMS}, we have the following a priori bound for solutions:

\begin{proposition}\label{apb}
Let ${\bf H}_0$ hold, $u\in\w$ be a solution of \eqref{p0}. Then, $u\in L^\infty(\Omega)$ with $\|u\|_\infty\le C$ for some $C=C(\|u\|)>0$.
\end{proposition}

\noindent
It is well known that, though solutions of \eqref{p0} can be very regular in $\Omega$, they fail to be smooth up to the boundary, even in simple cases (see \cite[Lemma 2.2]{IMS2}). So, a major role in fractional regularity theory is played by the following weighted H\"older spaces. Set $\ds(x)= \mathrm{dist}(x, \Omega^c)^s$, define
\[\cs = \Big\{u \in C^0(\overline{\Omega}): \frac{u}{\ds} \ \text{has a continuous extension to $\overline\Omega$}\Big\},\]
and for all $\alpha \in (0,1)$
\[C_s^{\alpha}(\overline{\Omega})= \Big\{u \in C^0(\overline{\Omega}): \frac{u}{\ds} \ \text{has a $\alpha$-H\"older continuous extension to $\overline\Omega$}\Big\},\]
whose norms are defined, respectively, by
\[\|u\|_{0,s}= \Big\|\frac{u}{\ds}\Big\|_{\infty}, \ \|u\|_{\alpha,s}= \|u\|_{0,s} + \sup_{x \neq y} \frac{|u(x)/\ds(x) - u(y)/\ds(y)|}{|x-y|^{\alpha}}.\]
The embedding $C_s^{\alpha}(\overline{\Omega}) \hookrightarrow C_s^0(\overline{\Omega})$ is compact for all
$\alpha \in (0,1)$. Unlike in $\w$, the positive cone $C_s^0(\overline{\Omega})_+$ of $C_s^0(\overline{\Omega})$ has a nonempty interior given by
\[\mathrm{int}(\cs_+)= \Big\{u \in \cs:\, \inf_{x\in\Omega}\frac{u(x)}{\ds(x)} > 0\Big\}\]
(equivalent characterization as in \cite[Lemma 5.1]{ILPS}). By Proposition \ref{apb} and \cite[Theorem 1.1]{IMS2} we have the following:

\begin{proposition}\label{reg}
Let ${\bf H}_0$ hold, $u\in\w$ be a solution of \eqref{p0}. Then, $u\in C^\alpha_s(\overline\Omega)$ for some $\alpha\in(0,s]$.
\end{proposition}

\noindent
The strong maximum principle and Hopf's lemma for the $p$-Laplacian have an analogue in the following result, see \cite[Theorems 1.2, 1.5]{DQ}:

\begin{proposition}\label{smp}
Let ${\bf H}_0$ hold, and $\eta_0\in L^\infty(\Omega)_+$ be s.t.\ for a.e.\ $x\in\Omega$ and all $t\ge 0$
\[f(x,t) \ge -\eta_0(x)t^{p-1}.\]
Then, for all $u\in\w_+\setminus\{0\}$ solution of \eqref{p0} we have $u\in{\rm int}(\cs_+)$.
\end{proposition}

\noindent
We define an energy functional for problem \eqref{p0} by setting for all $u\in\w$
\[\Phi_0(u) = \frac{\|u\|^p}{p}-\int_\Omega F(x,u)\,dx.\]
By ${\bf H}_0$, it is easily seen that $\Phi_0\in C^1(\w)$ with G\^ateaux derivative given for all $u,\varphi\in\w$ by
\[\langle\Phi_0'(u),\varphi\rangle = \langle\fpl u,\varphi\rangle-\int_\Omega f(x,u)\varphi\,dx.\]
So, $u\in\w$ is a solution of \eqref{p0} iff it is a critical point of $\Phi_0$, denoted $u\in K(\Phi_0)$. For all definitions and classical results of critical point theory, including elementary Morse theory, we refer to \cite{MMP}. Since we are going to work with truncations, we shall need the following equivalence principle for Sobolev and H\"older local minimizers of $\Phi_0$, respectively, see \cite[Theorem 1.1]{IMS3} (this is in fact a nonlocal, nonlinear version of the classical result of \cite{BN}):

\begin{proposition}\label{svh}
Let ${\bf H}_0$ hold, $u\in\w$. Then, the following are equivalent:
\begin{enumroman}
\item\label{svh1} there exists $\rho>0$ s.t.\ $\Phi_0(u+v)\ge\Phi_0(u)$ for all $v\in\w$, $\|v\|\le\rho$;
\item\label{svh2} there exists $\sigma>0$ s.t.\ $\Phi_0(u+v)\ge\Phi_0(u)$ for all $v\in\w\cap C_s^0(\overline\Omega)$, $\|v\|_{0,s}\le\sigma$.
\end{enumroman}
\end{proposition}

\noindent
Contrary to many works in this area, we are not going to use much of the spectral properties of the leading operator $\fpl$. We only recall that the principal eigenvalue $\lambda_1>0$ of $\fpl$ in $\w$ is characterized by
\beq\label{l1}
\lambda_1 = \inf_{u\in\w\setminus\{0\}}\frac{\|u\|^p}{\|u\|_p^p},
\eeq
the infimum being attained at a one-dimensional eigenspace. We denote $\hat u_1\in{\rm int}(\cs_+)$ the unique positive, $L^p$-normalized eigenfunction (see \cite{LL}). We will use the following technical lemma:

\begin{lemma}\label{wei}
Let $\xi_0\in L^\infty(\Omega)$ be s.t.\ $\xi_0\le\lambda_1$ in $\Omega$, $\xi_0\not\equiv\lambda_1$. Then, there exists $\sigma>0$ s.t.\ for all $u\in\w$
\[\|u\|^p-\int_\Omega\xi_0(x)|u|^p\,dx \ge \sigma\|u\|^p.\]
\end{lemma}
\begin{proof}
Equivalently, we prove that for all $u\in\w$, $\|u\|=1$
\[\|u\|^p-\int_\Omega\xi_0(x)|u|^p\,dx \ge \sigma.\]
Arguing by contradiction, assume that there exists a sequence $(u_n)$ in $\w$ s.t.\ $\|u_n\|=1$ for all $n\in\N$ and
\[\lim_n\Big[\|u_n\|^p-\int_\Omega\xi_0(x)|u_n|^p\,dx\Big] = 0.\]
Since $(u_n)$ is bounded, passing if necessary to a subsequence we have $u_n\rightharpoonup u$ in $\w$, $u_n\to u$ in $L^p(\Omega)$. By \eqref{l1} we have
\begin{align}\label{wei1}
0 &\le \|u\|^p-\lambda_1\|u\|_p^p \\
\nonumber&\le \|u\|^p-\int_\Omega\xi_0(x)|u|^p\,dx \\
\nonumber&\le \lim_n\Big[\|u_n\|^p-\int_\Omega\xi_0(x)|u_n|^p\,dx\Big] = 0.
\end{align}
Besides, since $u_n\to u$ in $L^p(\Omega)$ we have
\[\int_\Omega\xi_0(x)|u|^p\,dx = \lim_n\int_\Omega\xi_0(x)|u_n|^p\,dx = 1,\]
hence $u\neq 0$. So, $u$ is a principal eigenfunction. By simplicity of $\lambda_1$, there exists $\tau\neq 0$ s.t.\ $u=\tau\hat u_1$. Since $\hat u_1\in{\rm int}(\cs_+)$, we deduce $|u|>0$ in $\Omega$, so
\[\int_\Omega\xi_0(x)|u|^p\,dx < \lambda_1\|u\|_p^p,\]
against \eqref{wei1}.
\end{proof}

\noindent
We conclude this section by presenting a weak comparison result for positive sub-supersolutions of \eqref{p0}. This will play a crucial role in the proof of existence of extremal constant sign solutions (see Section \ref{sec4} below), but it also is of independent interest:

\begin{theorem}\label{com}
Let ${\bf H}_0$ hold and assume that
\[t\mapsto\frac{f(x,t)}{t^{p-1}}\]
is decreasing in $(0,\infty)$ for a.e.\ $x\in\Omega$. Let $u,v\in{\rm int}(\cs_+)$ be a subsolution and a supersolution, respectively, of \eqref{p0}. Then, $u\le v$ in $\Omega$.
\end{theorem}
\begin{proof}
Since $u,v\in{\rm int}(\cs_+)$, we can find $C>1$ s.t.\ in $\Omega$
\[\frac{1}{C} \le \frac{u}{\ds},\,\frac{v}{\ds} \le C,\]
hence $u/v,v/u\in L^\infty(\Omega)$. We argue by contradiction, assuming that $|\Omega_0|>0$, where
\[\Omega_0 = \big\{x\in\Omega: \ u(x)>v(x)\big\}.\]
Define $u_0,v_0\in L^p(\R^N)$, $\varphi\in L^1(\R^N)$ by setting
\[u_0=u\chi_{\Omega_0}, \ v_0=v\chi_{\Omega_0}, \ \varphi=(u^p-v^p)^+=u_0^p-v_0^p.\]
In the following lines, we will identify the functions $\varphi/u^{p-1}$, $\varphi/v^{p-1}$ with the $0$-extensions of such functions to $\R^N$. We aim at using $\varphi/u^{p-1}$, $\varphi/v^{p-1}$ as test functions in \eqref{p0}, so we need to check that these functions belong in $\w$. First we note that there exists $M>0$ s.t.\ in $\R^N$
\[0 \le \frac{\varphi}{u^{p-1}},\,\frac{\varphi}{v^{p-1}} \le Mu_0,\]
hence $\varphi/u^{p-1},\varphi/v^{p-1}\in L^p(\R^N)_+$ and both vanish in $\Omega^c$. Moreover, we claim that there exists $C>0$ s.t.\ for all $x,y\in\R^N$
\beq\label{com1}
\Big|\frac{\varphi(x)}{u^{p-1}(x)}-\frac{\varphi(y)}{u^{p-1}(y)}\Big|,\,\Big|\frac{\varphi(x)}{v^{p-1}(x)}-\frac{\varphi(y)}{v^{p-1}(y)}\Big| \le C\big(|u(x)-u(y)|+|v(x)-v(y)|\big).
\eeq
Indeed, fix $x,y\in\R^N$. By symmetry, we only consider the following cases:
\begin{itemize}[leftmargin=1cm]
\item[$(a)$] if $x,y\in\Omega_0$ and $u(x)>u(y)$, then by Lagrange's theorem we have
\begin{align*}
\Big|\frac{\varphi(x)}{u^{p-1}(x)}-\frac{\varphi(y)}{u^{p-1}(y)}\Big| &= \Big|u(x)-\frac{v^p(x)}{u^{p-1}(x)}-u(y)+\frac{v^p(y)}{u^{p-1}(y)}\Big| \\
&\le (u(x)-u(y))+\Big|\frac{v^p(x)}{u^{p-1}(x)}-\frac{v^p(y)}{u^{p-1}(x)}+\frac{v^p(y)}{u^{p-1}(x)}-\frac{v^p(y)}{u^{p-1}(y)}\Big| \\
&\le (u(x)-u(y))+\frac{|v^p(x)-v^p(y)|}{u^{p-1}(x)}+v^p(x)\frac{u^{p-1}(x)-u^{p-1}(y)}{u^{p-1}(x)u^{p-1}(y)} \\
&\le (u(x)-u(y))+p\frac{\max\{v^{p-1}(x),\,v^{p-1}(y)\}}{u^{p-1}(x)}|v(x)-v(y)| \\
&+ (p-1)v(y)\frac{\max\{u^{p-2}(x),\,u^{p-2}(y)\}}{u^{p-1}(x)}(u(x)-u(y)) \\
&\le p|u(x)-u(y)|+p|v(x)-v(y)|,
\end{align*}
while using the boundedness of $u/v$, $v/u$ we derive
\begin{align*}
\Big|\frac{\varphi(x)}{v^{p-1}(x)}-\frac{\varphi(y)}{v^{p-1}(y)}\Big| &=\Big|\frac{u^p(x)}{v^{p-1}(x)}-v(x)-\frac{u^p(y)}{v^{p-1}(y)}+v(y)\Big| \\
&\le |v(x)-v(y)|+\Big|\frac{u^p(x)}{v^{p-1}(x)}-\frac{u^p(y)}{v^{p-1}(x)}+\frac{u^p(y)}{v^{p-1}(x)}-\frac{u^p(y)}{v^{p-1}(y)}\Big| \\
&\le |v(x)-v(y)|+\frac{u^p(x)-u^p(y)}{v^{p-1}(x)}+u^p(y)\frac{|v^{p-1}(x)-v^{p-1}(y)|}{v^{p-1}(x)v^{p-1}(y)} \\
&\le |v(x)-v(y)|+p\frac{\max\{u^{p-1}(x),\,u^{p-1}(y)\}}{v^{p-1}(x)}(u(x)-u(y)) \\
&+ Cu(y)(p-1)\frac{\max\{v^{p-2}(x),\,v^{p-2}(y)\}}{v^{p-1}(x)}|v(x)-v(y)| \\
&\le |v(x)-v(y)|+C(u(x)-u(y))+C\frac{u^{p-1}(x)}{v^{p-1}(x)}|v(x)-v(y)| \\
&\le C|u(x)-u(y)|+C|v(x)-v(y)|;
\end{align*}
\item[$(b)$] if $x\in\Omega_0$, $y\notin\Omega_0$, then
\begin{align*}
\Big|\frac{\varphi(x)}{u^{p-1}(x)}-\frac{\varphi(y)}{u^{p-1}(y)}\Big| &= \frac{u^p(x)-v^p(x)}{u^{p-1}(x)} \\
&\le p\frac{\max\{u^{p-1}(x),\,v^{p-1}(x)\}}{u^{p-1}(x)}(u(x)-v(x)) \\
&= p\big[(u(x)-u(y))+(u(y)-v(y))+(v(y)-v(x))\big] \\
&\le p|u(x)-u(y)|+p|v(x)-v(y)|,
\end{align*}
and similarly
\[\Big|\frac{\varphi(x)}{v^{p-1}(x)}-\frac{\varphi(y)}{v^{p-1}(y)}\Big| \le C|u(x)-u(y)|+C|v(x)-v(y)|;\]
\item[$(c)$] if $x,y\notin\Omega_0$, finally, then
\[\varphi(x) = \varphi(y) = 0.\]
\end{itemize}
In all cases, \eqref{com1} holds. Hence, by integrating we have
\[\iint_{\R^N\times\R^N}\Big|\frac{\varphi(x)}{u^{p-1}(x)}-\frac{\varphi(y)}{u^{p-1}(y)}\Big|^p\,\frac{dx\,dy}{|x-y|^{N+ps}} \le C(\|u\|^p+\|v\|^p),\]
so $\varphi/u^{p-1}\in\w_+$. Similarly we see that $\varphi/v^{p-1}\in\w_+$. The next step consists in proving that for all $x,y\in\R^N$
\beq\label{com2}
j_p(v(x)-v(y))\Big[\frac{\varphi(x)}{v^{p-1}(x)}-\frac{\varphi(y)}{v^{p-1}(y)}\Big] \le j_p(u(x)-u(y))\Big[\frac{\varphi(x)}{u^{p-1}(x)}-\frac{\varphi(y)}{u^{p-1}(y)}\Big],
\eeq
where we have set $j_p(a)=|a|^{p-2}a$ for all $a\in\R$. First, we rephrase \eqref{com2} as
\[A+B \le C+D,\]
where
\[A = j_p(v(x)-v(y))\Big[\frac{u_0^p(x)}{v^{p-1}(x)}-\frac{u_0^p(y)}{v^{p-1}(y)}\Big], \ B = j_p(u(x)-u(y))\Big[\frac{v_0^p(x)}{u^{p-1}(x)}-\frac{v_0^p(y)}{u^{p-1}(y)}\Big],\]
\[C = j_p(v(x)-v(y))(v_0(x)-v_0(y)), \ D = j_p(u(x)-u(y))(u_0(x)-u_0(y)).\]
As above, we consider three cases:
\begin{itemize}[leftmargin=1cm]
\item[$(a)$] if $x,y\in\Omega_0$, then we apply a discrete Picone's inequality:
\[j_p(a-b)\Big[\frac{c^p}{a^{p-1}}-\frac{d}{b^{p-1}}\Big] \le |c-d|^p\]
for all $a,b>0$, $c,d\ge 0$ (see \cite[Proposition 2.2]{BS}), to get
\begin{align*}
A &= j_p(v(x)-v(y))\Big[\frac{u^p(x)}{v^{p-1}(x)}-\frac{u^p(y)}{v^{p-1}(y)}\Big] \\
&\le |u(x)-u(y)|^p = D,
\end{align*}
and similarly $B\le C$;
\item[$(b)$] if $x\in\Omega_0$, $y\notin\Omega_0$, then $v(y)/v(x) \ge u(y)/u(x)$, hence
\begin{align*}
A-C &= j_p(v(x)-v(y))\frac{u^p(x)-v^p(x)}{v^{p-1}(x)} \\
&= j_p\Big(1-\frac{v(y)}{v(x)}\Big)(u^p(x)-v^p(x)) \\
&\le j_p\Big(1-\frac{u(y)}{u(x)}\Big)(u^p(x)-v^p(x)) \\
&= j_p(u(x)-u(y))\frac{u^p(x)-v^p(x)}{u^{p-1}(x)} = D-B;
\end{align*}
\item[$(c)$] if $x,y\notin\Omega_0$, then
\[A = B = C = D = 0.\]
\end{itemize}
Integrating \eqref{com2}, we immediately get
\beq\label{com3}
\Big\langle\fpl v,\frac{\varphi}{v^{p-1}}\Big\rangle \le \Big\langle\fpl u,\frac{\varphi}{u^{p-1}}\Big\rangle.
\eeq
Now recall that $u$ and $v$ are a sub- and a supersolution, respectively, of \eqref{p0}, so testing with $\varphi/u^{p-1},\varphi/v^{p-1}\in\w_+$ and applying the monotonicity assumption we have
\begin{align*}
\Big\langle\fpl u,\frac{\varphi}{u^{p-1}}\Big\rangle &\le \int_\Omega f(x,u)\frac{\varphi}{u^{p-1}}\,dx \\
&= \int_{\Omega_0}\frac{f(x,u)}{u^{p-1}}(u^p-v^p)\,dx \\
&< \int_{\Omega_0}\frac{f(x,v)}{v^{p-1}}(u^p-v^p)\,dx \\
&= \int_\Omega f(x,v)\frac{\varphi}{v^{p-1}}\,dx \\
&\le \Big\langle\fpl v,\frac{\varphi}{v^{p-1}}\Big\rangle,
\end{align*}
against \eqref{com3}. Thus $u\le v$ in $\Omega$.
\end{proof}

\begin{remark}\label{rbo}
Theorem \ref{com} is a partial analogue for the fractional $p$-Laplacian of the classical results of \cite{BO,DS}. Similar results in the fractional setting were obtained in \cite{LPPS} for $p=2$, in \cite{BS} for any $p>1$ and a pure power reaction, and in \cite{MPV} for Robin boundary condition. In our case, we make a close connection to the regularity result of \cite{IMS2} in assuming that both $u,v\in{\rm int}(\cs_+)$, which allows for a simpler proof. We note, {\em en passant}, that by applying Theorem \ref{com} twice one can easily prove that, under the same monotonicity assumption, problem \eqref{p0} has at most one solution.
\end{remark}

\section{Constant sign solutions}\label{sec3}

\noindent
This section is devoted to the existence of positive and negative solutions of \eqref{dir}. Here we assume the following hypotheses on the perturbation $g$:

\begin{itemize}[leftmargin=1cm]
\item[${\bf H}_1$] $g:\Omega\times\R\to\R$ is a Carath\'eodory function, we set $\displaystyle G(x,t) = \int_0^t g(x,\tau)\,d\tau$ for all $(x,t)\in\R$, and
\begin{enumroman}
\item\label{h11} there exist $c_1>0$, $r\in(p,\p)$ s.t.\ for a.e.\ $x\in\R$ and all $t\in\R$
\[|g(x,t)|\le c_1(1+|t|^{r-1});\]
\item\label{h12} uniformly for a.e.\ $x\in\Omega$
\[\lim_{t\to\infty}\frac{G(x,t)}{t^p}=\infty;\]
\item\label{h13} there exist $c_2,\beta>0$, with $\displaystyle\max\Big\{q,\,\frac{N(r-p)}{ps}\Big\}<\beta<\p$ s.t.\ uniformly for a.e.\ $x\in\Omega$
\[\liminf_{t\to\infty}\frac{g(x,t)t-pG(x,t)}{t^\beta}\ge c_2;\]
\item\label{h14} there exist $\eta_1,\eta_2\in L^\infty(\Omega)_+$ s.t.\ $\eta_2\le\lambda_1$ in $\Omega$, $\eta_2\not\equiv\lambda_1$, and uniformly for a.e.\ $x\in\Omega$
\[-\eta_1(x)\le\liminf_{t\to 0}\frac{g(x,t)}{|t|^{p-2}t}\le\limsup_{t\to 0}\frac{g(x,t)}{|t|^{p-2}t}\le\eta_2(x);\]
\item\label{h15} there exists $\theta\in L^\infty(\Omega)_+$ s.t.\ $\theta\le\lambda_1$ in $\Omega$, $\theta\not\equiv\lambda_1$, and uniformly for a.e.\ $x\in\Omega$
\[\limsup_{t\to -\infty}\frac{G(x,t)}{|t|^p}\le\frac{\theta(x)}{p}.\]
\end{enumroman}
\end{itemize}
Hypothesis ${\bf H}_1$ \ref{h11} is a subcritical growth condition, useful in obtaining compactness properties for the energy functional. Hypothesis \ref{h12} forces for $g(x,\cdot)$ a $(p-1)$-superlinear growth at $\infty$, tempered by an asymptotic condition of Ambrosetti-Rabinowitz type \ref{h13} (this was first introduced in \cite{CM} for the Laplacian). By \ref{h14}, $g(x,\cdot)$ is $(p-1)$-linear at zero and by \ref{h15} it is at most $(p-1)$-linear at $-\infty$, thus exhibiting an asymmetric behavior. For simplicity, we assume in both cases that possible $(p-1)$-linear behaviors have no resonance with the principal eigenvalue in all of $\Omega$.

\begin{example}\label{eh1}
The following autonomous mapping $g\in C(\R)$ clearly satisfies ${\bf H}_1$:
\[g(t) = a|t|^{p-2}t+(t^+)^{r-1},\]
with $a\in(0,\lambda_1)$, $r\in(p,\p)$ (set $\beta=r$ in \ref{h13}).
\end{example}

\noindent
Fix $\lambda>0$ and set for all $(x,t)\in\Omega\times\R$
\[f_\lambda(x,t) = \lambda|t|^{q-2}t+g(x,t), \ F_\lambda(x,t) = \int_0^t f_\lambda(x,\tau)\,d\tau.\]
Clearly, by ${\bf H}_1$ we see that $f_\lambda$ satisfies ${\bf H}_0$. So, we can define an energy functional $\Phi_\lambda\in C^1(\w)$ for problem \eqref{dir} by setting for all $u\in\w$
\[\Phi_\lambda(u) = \frac{\|u\|^p}{p}-\int_\Omega F_\lambda(x,u)\,dx.\]
By ${\bf H}_1$ \ref{h14}, we easily see that $f_\lambda(\cdot,0)=0$ in $\Omega$, so $0\in K(\Phi_\lambda)$ for all $\lambda>0$, i.e., \eqref{dir} always admits the trivial solution.
\vskip2pt
\noindent
In order to detect constant sign solutions, we define two truncated energy functionals. Set for all $(x,t)\in\Omega\times\R$
\[f^\pm_\lambda(x,t) = f_\lambda(x,\pm t^\pm), \ F^\pm_\lambda(x,t) = \int_0^t f_\lambda^\pm(x,\tau)\,d\tau,\]
and for all $u\in\w$
\[\Phi^\pm_\lambda(u) = \frac{\|u\|^p}{p}-\int_\Omega F^\pm_\lambda(x,u)\,dx.\]
We first focus on positive solutions, starting with a crucial compactness property, see \cite[Definition 5.14 (b)]{MMP1}:

\begin{lemma}\label{cer}
Let ${\bf H}_1$ hold. Then, $\Phi^+_\lambda\in C^1(\w)$ satisfies the Cerami $(C)$-condition.
\end{lemma}
\begin{proof}
As in Section \ref{sec2} we see that $\Phi^+_\lambda\in C^1(\w)$ with derivative given for all $u,\varphi\in\w$ by
\[\langle(\Phi^+_\lambda)'(u),\varphi\rangle = \langle\fpl u,\varphi\rangle-\int_\Omega f^+_\lambda(x,u)\varphi\,dx.\]
Let $(u_n)$ be a sequence in $\w$ s.t.\ $(\Phi^+_\lambda(u_n))$ is bounded in $\R$ and $(1+\|u_n\|)(\Phi^+_\lambda)'(u_n)\to 0$ in $W^{-s,p'}(\Omega)$. Then, there exist $C>0$ and a sequence $(\eps_n)$ with $\eps_n\to 0^+$, s.t.\ for all $n\in\N$
\beq\label{cer1}
\Big|\frac{\|u_n\|^p}{p}-\int_\Omega F^+_\lambda(x,u_n)\,dx\Big| \le C
\eeq
and for all $\varphi\in\w$
\beq\label{cer2}
\Big|\langle\fpl u_n,\varphi\rangle-\int_\Omega f^+_\lambda(x,u_n)\varphi\,dx\Big| \le \frac{\eps_n\|\varphi\|}{1+\|u_n\|}.
\eeq
First we prove that
\beq\label{cer3}
u^-_n\to 0 \ \text{in $\w$.}
\eeq
Choose $\varphi=-u^-_n\in\w$ in \eqref{cer2}, then by Lemma \ref{pn} \ref{pn2} we have for all $n\in\N$
\begin{align*}
\|u^-_n\|^p &\le \langle\fpl u_n,-u^-_n\rangle \\
&\le \int_\Omega f^+_\lambda(x,u_n)(-u^-_n)\,dx+\frac{\eps_n\|u^-_n\|}{1+\|u_n\|} \le \eps_n,
\end{align*}
and the latter tends to $0$ as $n\to\infty$. Next we prove that
\beq\label{cer4}
(u^+_n) \ \text{is bounded in $\w$.}
\eeq
By \eqref{cer1} we have for all $n\in\N$
\[\|u_n\|^p-\int_\Omega pF^+_\lambda(x,u_n)\,dx \le Cp.\]
Besides, by inequality \eqref{dua} and Lemma \ref{pn} \ref{pn1} we have for all $n\in\N$
\[\langle\fpl u_n,u^+_n\rangle \le \|u_n\|^{p-1}\|u^+_n\| \le \|u_n\|^p,\]
which along with \eqref{cer2} with $\varphi=u^+_n\in\w$ yields
\[-\|u_n\|^p+\int_\Omega f^+_\lambda(x,u_n)u^+_n\,dx \le \eps_n.\]
Adding the inequalities above and recalling the definition of $f^+_\lambda$, we have
\[\int_\Omega\big[g(x,u^+_n)u^+_n-pG(x,u^+_n)\big]\,dx \le \lambda\Big(\frac{p}{q}-1\Big)\|u^+_n\|_q^q+C.\]
By ${\bf H}_1$ \ref{h13} we can find $K>0$ s.t.\ for a.e.\ $x\in\Omega$ and all $t>K$
\[g(x,t)t-pG(x,t) \ge \frac{c_2}{2}t^\beta.\]
Also recalling ${\bf H}_1$ \ref{h11}, we can find $C>0$ s.t.\ for all $n\in\N$
\[\int_\Omega\big[g(x,u^+_n)u^+_n-pG(x,u^+_n)\big]\,dx \ge \frac{c_2}{2}\|u^+_n\|_\beta^\beta-C.\]
By the previous relations and H\"older's inequality, we have
\begin{align*}
\|u^+_n\|_\beta^\beta &\le C\big(\|u^+_n\|_q^q+1\big) \\
&\le C\Big[\int_\Omega(u^+_n)^\beta\,dx\Big]^\frac{q}{\beta}|\Omega|^{1-\frac{q}{\beta}}+C \\
&\le C\big(\|u^+_n\|_\beta^q+1),
\end{align*}
which by $q<\beta$ implies that $(u^+_n)$ is bounded in $L^\beta(\Omega)$, and hence in $L^q(\Omega)$. In ${\bf H}_1$ \ref{h11} we may assume $\beta\le r<\p$, so we can find $\tau\in[0,1)$ s.t.\
\[\frac{1}{r} = \frac{1-\tau}{\beta}+\frac{\tau}{\p}.\]
By the interpolation inequality, boundedness of $(u^+_n)$ in $L^\beta(\Omega)$, and the embedding $\w\hookrightarrow L^{\p}(\Omega)$ we have
\[\|u^+_n\|_r \le \|u^+_n\|_\beta^{1-\tau}\|u^+_n\|_{\p}^\tau \le C\|u^+_n\|^\tau.\]
Test \eqref{cer2} with $\varphi=u^+_n\in\w$ and apply Lemma \ref{pn} \ref{pn2} to get
\begin{align}\label{cer5}
\|u^+_n\|^p &\le \lambda\|u^+_n\|_q^q+\int_\Omega g(x,u^+_n)u^+_n\,dx+\eps_n \\
\nonumber&\le \int_\Omega c_1\big[u^+_n+(u^+_n)^r\big]\,dx+C \\
\nonumber &\le C\big(1+\|u^+_n\|_1+\|u^+_n\|_r^r\big) \\
\nonumber &\le C\big(1+\|u^+_n\|+\|u^+_n\|^{\tau r}\big).
\end{align}
We note that, by ${\bf H}_1$ \ref{h13},
\begin{align*}
\frac{1}{r} &< (1-\tau)\frac{ps}{N(r-p)}+\tau\frac{N-ps}{Np} \\
&= \frac{ps}{N(r-p)}+\tau\frac{Nr-Np-psr}{Np(r-p)},
\end{align*}
which by $r<\p$ implies
\[\frac{\tau r}{p}\,\frac{Nr-Np-psr}{N(r-p)} > \frac{Nr-Np-psr}{N(r-p)},\]
and hence $\tau r<p$. So, from \eqref{cer5} we see that $(u^+_n)$ is bounded in $\w$.
\vskip2pt
\noindent
By \eqref{cer3}, \eqref{cer4} $(u_n)$ is bounded in $\w$. Passing to a subsequence, we may assume that $u_n\rightharpoonup u$ in $\w$, $u_n\to u$ in $L^r(\Omega)$. Testing \eqref{cer2} with $\varphi=u_n-u\in\w$, and applying H\"older's inequality, we have
\begin{align*}
\langle\fpl u_n,u_n-u\rangle &\le \lambda\int_\Omega(u^+_n)^{q-1}(u_n-u)\,dx+\int_\Omega g(x,u^+_n)(u_n-u)\,dx+\frac{\eps_n\|u_n-u\|}{1+\|u_n\|} \\
&\le \lambda\|u^+_n\|_q^{q-1}\|u_n-u\|_q+C\big(\|u_n-u\|_1+\|u^+_n\|_r^{r-1}\|u_n-u\|_r+\eps_n\big),
\end{align*}
and the latter tends to $0$ as $n\to\infty$. By the $(S)_+$-property of $\fpl$, we finally have $u_n\to u$ in $\w$. Thus, $\Phi_\lambda^+$ satisfies $(C)$.
\end{proof}

\noindent
Now we can prove the existence of two positive solutions for $\lambda>0$ small enough:

\begin{lemma}\label{pos}
Let ${\bf H}_1$ hold. Then, there exists $\lambda^*>0$ s.t.\ for all $\lambda\in(0,\lambda^*)$ problem \eqref{dir} has at least two positive solutions $u_+,v_+\in{\rm int}(\cs_+)$.
\end{lemma}
\begin{proof}
Fix $\lambda>0$ (to be better determined later). We will seek the first positive solution by applying the mountain pass theorem. First, we claim that there exists $\rho>0$ s.t.\
\beq\label{pos1}
\inf_{\|u\|=\rho}\Phi^+_\lambda(u) = m_+ >0.
\eeq
Indeed, by ${\bf H}_1$ \ref{h14} and Lemma \ref{wei}, there exists $\sigma>0$ s.t.\ for all $u\in\w$
\[\|u\|^p-\int_\Omega\eta_2(x)|u|^p\,dx \ge \sigma\|u\|^p.\]
Now fix $\eps\in(0,\sigma\lambda_1)$. By ${\bf H}_1$ \ref{h11} \ref{h14} we can find $C_\eps>0$ s.t.\ for a.e.\ $x\in\Omega$ and all $t\ge 0$
\[G(x,t) \le \frac{\eta_2(x)+\eps}{p}t^p+C_\eps t^r.\]
Set $\sigma'=\sigma-\eps/\lambda_1>0$. For all $u\in\w$ we have $0\le u^+\le|u|$ in $\Omega$, so by the estimates above, \eqref{l1}, and the embeddings of $\w$ we have
\begin{align*}
\Phi^+_\lambda(u) &\ge \frac{\|u\|^p}{p}-\frac{\lambda}{q}\|u^+\|_q^q-\int_\Omega\Big[\frac{\eta_2(x)+\eps}{p}(u^+)^p+C_\eps(u^+)^r\Big]\,dx \\
&\ge \frac{1}{p}\Big[\|u\|^p-\int_\Omega\eta_2(x)|u|^p\,dx\Big]-\frac{\lambda}{q}\|u\|_p^q|\Omega|^\frac{p-q}{p}-C_\eps\|u\|_r^r-\frac{\eps}{p}\|u\|_p^p \\
&\ge \frac{\sigma'}{p}\|u\|^p-\frac{\lambda|\Omega|^\frac{p-q}{p}}{q\lambda_1^\frac{q}{p}}\|u\|^q-C\|u\|^r = h(\|u\|)\|u\|^p,
\end{align*}
where for all $t>0$ we have set
\[h(t) = \frac{\sigma'}{p}-\frac{\lambda|\Omega|^{1-\frac{q}{p}}}{q\lambda_1^\frac{q}{p}}t^{q-p}-Ct^{r-p}.\]
Clearly, we have $h\in C^1(0,\infty)$, $h(t)\to -\infty$ as $t\to 0,\,\infty$ (recall that $q<p<r$). So there is $\rho>0$ s.t.\
\[h(\rho) = \max_{t>0}h(t).\]
We can detect $\rho>0$ by setting $h'(\rho)=0$, which gives
\[\rho = \left[\frac{\lambda|\Omega|^\frac{q-p}{p}(p-q)}{Cq\lambda_1^\frac{q}{p}(r-p)}\right]^\frac{1}{r-q} > 0.\]
In turn, that implies
\[h(\rho) = \frac{\sigma'}{p}-\left[\frac{\lambda|\Omega|^\frac{q-p}{p}}{q\lambda_1^\frac{q}{p}}\right]^\frac{r-p}{r-q}\Big[\frac{p-q}{C(r-p)}\Big]^\frac{q-p}{r-q}-C^\frac{p-q}{r-q}\left[\frac{\lambda|\Omega|^\frac{q-p}{p}(p-q)}{q(r-p)\lambda_1^\frac{q}{p}}\right]^\frac{r-p}{r-q},\]
and the latter tends to $\sigma'/p>0$ as $\lambda\to 0^+$. So there exists $\lambda^*>0$ s.t.\ for all $\lambda\in(0,\lambda^*)$
\[\inf_{\|u\|=\rho}\Phi^+_\lambda(u) \ge h(\rho)\rho^p >0,\]
which proves \eqref{pos1}. Let $\hat u_1\in{\rm int}(\cs_+)$ be as in Section \ref{sec2}, then we have
\beq\label{pos2}
\lim_{\tau\to\infty}\Phi^+_\lambda(\tau\hat u_1) = -\infty.
\eeq
Indeed, by ${\bf H}_1$ \ref{h11} \ref{h12}, for any $M>0$ we can find $C_M>0$ s.t.\ for a.e.\ $x\in\Omega$ and all $t\ge 0$
\[G(x,t) \ge Mt^p-C_M.\]
So, for all $\tau>0$ we have
\begin{align*}
\Phi^+_\lambda(\tau\hat u_1) &\le \frac{\tau^p}{p}\|\hat u_1\|^p-\frac{\tau^q\lambda}{q}\|\hat u_1\|_q^q-\int_\Omega\big(Mt^p(\hat u_1)^p-C_M\big)\,dx \\
&\le \Big(\frac{\lambda_1}{p}-M\Big)\tau^p-\frac{\tau^q\lambda}{q}\|\hat u_1\|_q^q-C_M|\Omega|,
\end{align*}
an the latter tends to $-\infty$ as $\tau\to\infty$, as soon as we choose $M>\lambda_1/p$. By \eqref{pos1}, \eqref{pos2} $\Phi^+_\lambda$ exhibits a mountain pass geometry, while by Lemma \ref{cer} it satisfies $(C)$. By the mountain pass theorem (see for instance \cite[Theorem 5.40]{MMP}) there exists $u_+\in K(\Phi^+_\lambda)$ s.t.\
\[\Phi^+_\lambda(u_+) \ge m_+.\]
By \eqref{pos1} we have $u_+\neq 0$. Testing $(\Phi^+_\lambda)'(u_+)=0$ with $-u_+^-\in\w$ and recalling Lemma \ref{pn} \ref{pn2}, we have
\begin{align*}
\|u_+^-\|^p &\le \langle\fpl u_+,-u_+^-\rangle \\
&= \int_\Omega f_\lambda^+(x,u_+)(-u_+^-)\,dx = 0,
\end{align*}
so $u_+\in\w_+\setminus\{0\}$. That in turn implies that $u_+$ solves \eqref{dir}. Since $f_\lambda$ satisfies ${\bf H}_0$, by Proposition \ref{reg} we have $u_+\in C^\alpha_s(\overline\Omega)$. Further, by ${\bf H}_1$ \ref{h12} \ref{h14} we can find $C>0$ s.t.\ for a.e.\ $x\in\Omega$ and all $t\ge 0$
\[f_\lambda(x,t) \ge -Ct^{p-1}.\]
By Proposition \ref{smp} we have $u_+\in{\rm int}(\cs_+)$.
\vskip2pt
\noindent
Now we seek a second positive solution. By ${\bf H}_1$ \ref{h14} we can find $\delta,c>0$ s.t.\ for a.e.\ $x\in\Omega$ and all $t\in[0,\delta]$
\[G(x,t) \ge -ct^p.\]
Since $\hat u_1\in{\rm int}(\cs_+)$, for all $\tau>0$ small enough we have $0<\tau\hat u_1\le\delta$ in $\Omega$, so
\begin{align*}
\Phi^+_\lambda(\tau\hat u_1) &\le \frac{\tau^p}{p}\|\hat u_1\|^p-\frac{\lambda\tau^q}{q}\|\hat u_1\|_q^q+c\tau^p\|\hat u_1\|_p^p \\
&= \Big(\frac{\lambda_1}{p}+c\Big)\tau^p-\frac{\lambda}{q}\tau^q,
\end{align*}
and the latter is negative for all $\tau>0$ small enough. So, by \eqref{pos1} we have
\beq\label{pos3}
\inf_{\|u\|\le\rho}\Phi^+_\lambda(u) < 0 <m_+.
\eeq
Since $\Phi^+_\lambda\in C^1(\w)$ is sequentially weakly l.s.c., there exists $v_+\in\overline B_\rho(0)$ s.t.\
\[\Phi^+_\lambda(v_+) = \inf_{\|u\|\le\rho}\Phi^+_\lambda(u).\]
By \eqref{pos1} and \eqref{pos3} we have $\|v_+\|<\rho$, so $v_+\in K(\Phi^+_\lambda)$ is a local minimizer of $\Phi^+_\lambda$ (not a global one, due to \eqref{pos2}). Besides, since
\[\Phi^+_\lambda(v_+) < 0 < m_+ \le \Phi^+_\lambda(u_+),\]
we deduce $v_+\neq 0,u_+$. Arguing as above, we conclude that $v_+\in{\rm int}(\cs_+)$ solves \eqref{dir} and complete the proof.
\end{proof}

\noindent
The existence of a negative solution is achieved by combining truncations and direct methods. Notably, this holds for any $\lambda>0$:

\begin{lemma}\label{neg}
Let ${\bf H}_1$ hold. Then, for all $\lambda>0$ problem \eqref{dir} has at least one negative solution $u_-\in -{\rm int}(\cs_+)$.
\end{lemma}
\begin{proof}
Fix $\lambda>0$ and recall the definition of $\Phi^-_\lambda\in C^1(\w)$. We prove first that $\Phi^-_\lambda$ is coercive. Indeed, by ${\bf H}_1$ \ref{h11} \ref{h15}, for any $\eps>0$ we can find $C_\eps>0$ s.t.\ for a.e.\ $x\in\Omega$ and all $t\le 0$
\[G(x,t) \le \frac{\theta(x)+\eps}{p}|t|^p+C_\eps.\]
Besides, by Lemma \ref{wei} we can find $\sigma>0$ s.t.\ for all $u\in\w$
\[\|u\|^p-\int_\Omega\theta(x)|u|^p\,dx \ge \sigma\|u\|^p.\]
So, recalling that $0\le u^-\le|u|$ in $\Omega$ and using \eqref{l1}, we have
\begin{align*}
\Phi^-_\lambda(x) &\ge \frac{\|u\|^p}{p}-\frac{\lambda}{q}\|u^-\|_q^q-\int_\Omega\Big[\frac{\theta(x)+\eps}{p}(u^-)^p+C_\eps\Big]\,dx \\
&\ge \Big(\sigma-\frac{\eps}{\lambda_1}\Big)\frac{\|u\|^p}{p}-C\|u\|^q-C,
\end{align*}
and the latter tends to $\infty$ as $\|u\|\to\infty$, as soon as we choose $\eps<\sigma\lambda_1$. Also, $\Phi^-_\lambda$ is sequentially weakly l.s.c.\ in $\w$, so there exists $u_-\in\w$ s.t.\
\beq\label{neg1}
\Phi^-_\lambda(u_-) = \inf_{u\in\w}\Phi^-_\lambda(u) = m_-.
\eeq
By ${\bf H}_1$ \ref{h14} we can find $c,\delta>0$ s.t.\ for a.e.\ $x\in\Omega$ and all $t\in[-\delta,0]$
\[G(x,t) \ge -c|t|^p.\]
Since $\hat u_1\in{\rm int}(\cs_+)$, for all $\tau>0$ small enough we have $-\delta<-\tau\hat u_1<0$ in $\Omega$, so
\begin{align*}
\Phi^-_\lambda(-\tau\hat u_1) &\le \frac{\tau^p}{p}\|\hat u_1\|^p-\frac{\lambda\tau^q}{q}\|\hat u_1\|_q^q+c\tau^p\|\hat u_1\|_p^p \\
&= \Big(\frac{\lambda_1}{p}+c\Big)\tau^p-\frac{\lambda\tau^q}{q}\|\hat u_1\|_q^q,
\end{align*}
and the latter is negative for all $\tau>0$ small enough. So we deduce $m_-<0$, hence by \eqref{neg1} we have $u_-\neq 0$. Testing $(\Phi^-_\lambda)'(u_-)=0$ with $u_-^+\in\w$ and recalling Lemma \ref{pn} \ref{pn2}, we have
\begin{align*}
\|u_-^+\|^p &\le \langle\fpl u_-,u_-^+\rangle \\
&= \int_\Omega f^-_\lambda(x,u_-)u_-^+\,dx = 0,
\end{align*}
so $u_-\in -\w_+\setminus\{0\}$. Arguing as in the proof of Lemma \ref{pos} and applying Propositions \ref{reg} and \ref{smp}, we see that $u_-\in -{\rm int}(\cs_+)$ is a negative solution of \eqref{dir}.\end{proof}

\noindent
Combining Lemmas \ref{pos} and \ref{neg}, we achieve our result on constant sign solutions:

\begin{theorem}\label{css}
Let ${\bf H}_1$ hold. Then, there exists $\lambda^*>0$ s.t.\ for all $\lambda\in(0,\lambda^*)$ problem \eqref{dir} has at least two positive solutions $u_+,v_+\in{\rm int}(\cs_+)$ and a negative solution $u_-\in -{\rm int}(\cs_+)$.
\end{theorem}

\begin{remark}\label{rcs}
We briefly outline that multiple constant sign solutions could be ensured under an alternative set of assumptions involving asymmetric reactions (see for instance \cite{IMP}). In particular, the pure power term $|u|^{q-2}u$ can be replaced by any Carath\'eodory mapping $h:\Omega\times\R\to\R$ with $(p-1)$-sublinear growth at $\pm\infty$ and satisfying a kind of reverse Ambrosetti-Rabinowitz condition at $0$. Moreover, the subcritical growth condition ${\bf H}_1$ \ref{h11} on $g(x,\cdot)$ can be weakened to a 'quasi-critical' one, namely, one may assume
\[\lim_{t\to\infty}\frac{g(x,t)}{t^{\p-1}} = 0 \ \text{uniformly for a.e.\ $x\in\Omega$.}\]
In such a case, however, a quasi-monotonicity condition must be required for the whole reaction $f_\lambda$ in order to retrieve the $(C)$-condition.
\end{remark}

\section{Extremal constant sign solutions and nodal solution}\label{sec4}

\noindent
In this section we get more precise information on constant sign solutions of \eqref{dir}, proving the existence of a smallest positive and a biggest negative solution, then we exploit such information to detect a nodal solution. To do so, we need to strengthen a bit our hypotheses on the perturbation $g$:
\begin{itemize}[leftmargin=1cm]
\item[${\bf H}_2$] $g:\Omega\times\R\to\R$ is a Carath\'eodory function, we set $\displaystyle G(x,t) = \int_0^t g(x,\tau)\,d\tau$ for all $(x,t)\in\R$, and
\begin{enumroman}
\item\label{h21} there exist $c_1>0$, $r\in(p,\p)$ s.t.\ for a.e.\ $x\in\R$ and all $t\in\R$
\[|g(x,t)|\le c_1(1+|t|^{r-1});\]
\item\label{h22} uniformly for a.e.\ $x\in\Omega$
\[\lim_{t\to\infty}\frac{G(x,t)}{t^p}=\infty;\]
\item\label{h23} there exist $c_2,\beta>0$, with $\displaystyle\max\Big\{q,\,\frac{N(r-p)}{ps}\Big\}<\beta<\p$ s.t.\ uniformly for a.e.\ $x\in\Omega$
\[\liminf_{t\to\infty}\frac{g(x,t)t-pG(x,t)}{t^\beta}\ge c_2;\]
\item\label{h24} uniformly for a.e.\ $x\in\Omega$
\[\lim_{t\to 0}\frac{g(x,t)}{|t|^{p-2}t}=0;\]
\item\label{h25} there exists $\theta\in L^\infty(\Omega)_+$ s.t.\ $\theta\le\lambda_1$ in $\Omega$, $\theta\not\equiv\lambda_1$, and uniformly for a.e.\ $x\in\Omega$
\[\limsup_{t\to -\infty}\frac{G(x,t)}{|t|^p}\le\frac{\theta(x)}{p};\]
\item\label{h26} there exist $\delta_1>0$ s.t.\ for a.e.\ $x\in\Omega$ and all $|t|\le\delta_1$
\[g(x,t)t\ge 0.\]
\end{enumroman}
\end{itemize}
Clearly ${\bf H}_2$ \ref{h21} -- \ref{h25} imply ${\bf H}_1$, so all results of Sections \ref{sec2} and \ref{sec3} still hold. In addition, we assume that $g(x,\cdot)$ is $(p-1)$-superlinear at $0$ (see \ref{h24}) and satisfies a local sign condition near zero (see \ref{h26}).

\begin{example}\label{eh2}
The following autonomous mapping $g\in C(\R)$ satisfies ${\bf H}_2$:
\[g(t) = \begin{cases}
|t|^{\gamma-2}t & \text{if $t< -1$} \\
|t|^{r-2}t & \text{if $t\ge -1$,}
\end{cases}\]
with $1<\gamma<p<r<\p$ (set $\beta=r$ in \ref{h23}).
\end{example}

\noindent
Taking $\lambda>0$ even smaller if necessary, problem \eqref{dir} admits extremal constant sign solutions. Unlike in \cite{FI} (where the reaction is $(p-1)$-linear at $0$ without resonance with the principal eigenvalue), the result is obtained by constructing a sub-supersolution pair by means of auxiliary problems and using the comparison result of Theorem \ref{com}:

\begin{lemma}\label{ext}
Let ${\bf H}_2$ hold. Then, there exists $\lambda_*>0$ s.t.\ for all $\lambda\in(0,\lambda_*)$ problem \eqref{dir} admits
\begin{enumroman}
\item\label{ext1} a smallest positive solution $w_+\in{\rm int}(\cs_+)$, $\|w_+\|_\infty\le\delta_1$;
\item\label{ext2} a biggest negative solution $w_-\in -{\rm int}(\cs_+)$, $\|w_-\|_\infty\le\delta_1$.
\end{enumroman}
\end{lemma}
\begin{proof}
We prove \ref{ext1}. First we consider the following torsion problem:
\beq\label{ext3}
\begin{cases}
\fpl v = 1 & \text{in $\Omega$} \\
v = 0 & \text{in $\Omega^c$.}
\end{cases}
\eeq
By direct variational methods (minimization) and Proposition \ref{smp}, we see that \eqref{ext3} has a unique solution $v\in{\rm int}(\cs_+)$. Fix $\eps\in(0,\|v\|_\infty^{1-p})$, then by ${\bf H}_2$ \ref{h21} \ref{h24} we can find $C_\eps>0$ s.t.\ for a.e.\ $x\in\Omega$ and all $t\ge 0$
\[g(x,t) \le \eps t^{p-1}+C_\eps t^{r-1}.\]
We claim that there exists $\lambda_*>0$ with the following property: for all $\lambda\in(0,\lambda_*)$ there is $\tau\in(0,\delta_1/\|v\|_\infty)$ s.t.\
\beq\label{ext4}
\lambda\|\tau v\|_\infty^{q-1}+\eps\|\tau v\|_\infty^{p-1}+C_\eps\|\tau v\|_\infty^{r-1} < \tau^{p-1}.
\eeq
Arguing by contradiction, let $(\lambda_n)$ be a sequence s.t.\ $\lambda_n\to 0^+$ and for all $n\in\N$, $\tau\in(0,\delta_1/\|v\|_\infty)$
\[\tau^{p-1} \le \lambda_n\|\tau v\|_\infty^{q-1}+\eps\|\tau v\|_\infty^{p-1}+C_\eps\|\tau v\|_\infty^{r-1}.\]
Then, letting $n\to\infty$ and dividing by $\tau^{p-1}>0$ we have
\[1 \le \eps\|v\|_\infty^{p-1}+C_\eps\tau^{r-p}\|v\|_\infty^{r-1}.\]
Now, letting $\tau\to 0^+$ and recalling that $r>p$ we get
\[1 \le \eps\|v\|_\infty^{p-1},\]
a contradiction. So \eqref{ext4} is achieved. Now fix $\lambda\in(0,\lambda_*)$, $\tau\in(0,\delta_1/\|v\|_\infty)$ satisfying \eqref{ext4}, and set
\[\overline{u} = \tau v \in{\rm int}(\cs_+).\]
Then, by \eqref{ext3} and the estimate on $f$ we have weakly in $\Omega$
\begin{align*}
\fpl\overline{u} &= \tau^{p-1} \\
&> \lambda\|\overline{u}\|_\infty^{q-1}+\eps\|\overline{u}\|_\infty^{p-1}+C_\eps\|\overline{u}\|_\infty^{r-1} \\
&\ge \lambda\overline{u}^{q-1}+g(x,\overline{u}),
\end{align*}
i.e., $\overline{u}\in{\rm int}(\cs_+)$ is a (strict) supersolution of \eqref{dir} satisfying $0<\overline{u}\le\delta_1$ in $\Omega$.
\vskip2pt
\noindent
For all $k\in\N$ set $\underline{u}_k=\hat u_1/k\in{\rm int}(\cs_+)$ (with $\hat u_1$ defined as in Section \ref{sec2}). Clearly, $\underline{u}_k\to 0$ uniformly in $\overline\Omega$, so for all $k\in\N$ big enough we have $\underline{u}_k<\overline{u}$ (in particular, $0<\underline{u}_k<\delta_1$) in $\Omega$, and $\lambda_1\underline{u}_k^{p-q}<\lambda$ in $\Omega$. By ${\bf H}_2$ \ref{h26} and the inequalities above, we have weakly in $\Omega$
\begin{align*}
\fpl\underline{u}_k &= \frac{\lambda_1}{k^{p-1}}\hat u_1^{p-1} \\
&= \lambda_1\underline{u}_k^{p-1} \\
&< \lambda\underline{u}_k^{q-1}+g(x,\underline{u}_k).
\end{align*}
So, for all $k\in\N$ big enough $\underline{u}_k\in{\rm int}(\cs_+)$ is a (strict) subsolution of \eqref{dir} s.t.\ $\underline{u}_k<\overline{u}$ in $\Omega$, namely $(\underline{u}_k,\overline{u})$ is a sub-supersolution pair of \eqref{dir}. By Proposition \ref{sss}, the set
\[\mathcal{S}(\underline{u}_k,\overline{u}) = \big\{w\in\w: \ \text{$w$ is a solution of \eqref{dir}, $\underline{u}_k\le w\le\overline{u}$ in $\Omega$}\big\}\]
has a smallest element $w_k\in\w$. By Propositions \ref{reg} and \ref{smp} we have $w_k\in{\rm int}(\cs_+)$. The sequence $(w_k)$ is relatively compact in $\w$. Indeed, for all $k\in\N$ we have $w_k\in\mathcal{S}(0,\overline{u})$, and the latter is a compact set in $\w$ (Proposition \ref{sss} again). Thus, passing to a subsequence we have $w_k\to w_+$ in $\w$, $w_k\to w_+$ in $L^p(\Omega)$, and $w_k(x)\to w_+(x)$ for a.e.\ $x\in\Omega$ (in particular, $0\le w_+\le\delta_1$ in $\Omega$). We claim that
\beq\label{ext5}
w_+ \neq 0.
\eeq
We argue by contradiction, assuming that $w_k\to 0$ in $\w$. Again we consider an auxiliary problem:
\beq\label{ext6}
\begin{cases}
\fpl\hat v = \lambda(\hat v^+)^{q-1} & \text{in $\Omega$} \\
\hat v = 0 & \text{in $\Omega^c$,}
\end{cases}
\eeq
with $\lambda\in(0,\lambda_*)$ as above. Since $q<p$, by direct variational methods and Proposition \ref{smp} we see that \eqref{ext6} has a solution $\hat v\in{\rm int}(\cs_+)$. By \cite[Remark 3.6]{FI}, passing to a subsequence we also have $w_k\to 0$ in $\cs$, in particular $w_k\to 0$ uniformly in $\overline\Omega$. So, let $k\in\N$ be large enough s.t.\ $0<w_k<\delta_1$ in $\Omega$. By ${\bf H}_2$ \ref{h26} we have weakly in $\Omega$
\[\fpl w_k = \lambda w_k^{q-1}+g(x,w_k) \ge \lambda w_k^{q-1},\]
i.e., $w_k\in{\rm int}(\cs_+)$ is a supersolution of \eqref{ext6}. Clearly, the mapping
\[t\mapsto\frac{\lambda}{t^{p-q}}\]
is decreasing in $(0,\infty)$, so by Theorem \ref{com} we have $\hat v\le w_k$ in $\Omega$. Letting $k\to\infty$ we get $\hat v\le 0$ in $\Omega$, a contradiction. Thus, \eqref{ext5} is proved.
\vskip2pt
\noindent
By strong convergence and \eqref{ext5}, we see that $w_+\in\w_+\setminus\{0\}$ solves \eqref{dir}, hence as above we deduce $w_+\in{\rm int}(\cs_+)$. Besides, from $w_+\le\overline{u}$ we deduce that in $\Omega$
\[0 < w_+ \le \delta_1.\]
We prove now that $w_+$ is the smallest positive solution of \eqref{dir}. Let $u\in\w_+\setminus\{0\}$ be another positive solution of \eqref{dir}, then $u\in{\rm int}(\cs_+)$. So we can find $k\in\N$ s.t.\ in $\Omega$ we have
\[\underline{u}_k = \frac{\hat u_1}{k} \le u.\]
Set
\[\hat u = \min\{u,\overline{u}\}.\]
By \cite[Lemma 3.1]{FI}, $\hat u\in\w$ is a supersolution of \eqref{dir}, so $(\underline{u}_k,\hat u)$ is a sub-supersolution pair. By Proposition \ref{sss} there exists a solution
\[v\in\mathcal{S}(\underline{u}_k,\hat u) \subseteq \mathcal{S}(\underline{u}_k,\overline{u}).\]
In particular, in $\Omega$ we have
\[w_k \le v \le \hat u \le u.\]
Letting $k\to\infty$, we have $w_+\le u$ in $\Omega$.
\vskip2pt
\noindent
The existence \ref{ext2} of a biggest negative solution $w_-\in -{\rm int}(\cs_+)$ s.t.\ $\|w_-\|_\infty\le\delta_1$ is proved in a similar way.
\end{proof}

\begin{remark}\label{rex}
For alternative hypotheses to ${\bf H}_2$ see \cite{FI} where (as already mentioned) extremal constant sign solutions are detected for $(p-1)$-linear reactions at $0$. Also, in \cite{MMP1} (dealing with the local case $s=1$) a different set of assumptions is proposed to find a biggest negative solution, namely, a $(p-1)$-linear behavior of $g(x,\cdot)$ near $0$ with a global sign condition.
\end{remark}

\noindent
In what follows, we seek a fourth nontrivial solution of \eqref{dir} under hypotheses ${\bf H}_2$, for $\lambda>0$ small enough. Set
\beq\label{tla}
\tilde\lambda = \min\{\lambda^*,\lambda_*\} > 0,
\eeq
with $\lambda^*>0$ as in Theorem \ref{css} and $\lambda_*>0$ as in Lemma \ref{ext}. Without loss of generality we may assume that for all $\lambda\in(0,\tilde\lambda)$ that $u_\pm\in\pm{\rm int}(\cs_+)$ are the extremal constant sign solutions given by Lemma \ref{ext}, in particular $v_+\ge u_+$ in $\Omega$. Set for all $(x,t)\in\Omega\times\R$
\[\kappa(x,t) = \begin{cases}
u_-(x) & \text{if $t\le u_-(x)$} \\
t & \text{if $u_-(x)<t<u_+(x)$} \\
u_+(x) & \text{if $t\ge u_+(x)$.}
\end{cases}\]
Accordingly, for all $\lambda>0$ set
\[\tilde{f}_\lambda(x,t) = \lambda|\kappa(x,t)|^{q-2}\kappa(x,t)+g(x,\kappa(x,t)),\]
\[\tilde{F}_\lambda(x,t) = \int_0^t\tilde{f}_\lambda(x,\tau)\,d\tau.\]
Further, set for all $u\in\w$
\[\tilde{\Phi}_\lambda(u) = \frac{\|u\|^p}{p}-\int_\Omega\tilde{F}_\lambda(x,u)\,dx.\]

\begin{lemma}\label{til}
Let ${\bf H}_2$ hold. Then,
\begin{enumroman}
\item\label{til1} $\tilde\Phi_\lambda\in C^1(\w)$ is coercive and satisfies the Palais-Smale $(PS)$-condition;
\item\label{til2} if $u\in K(\tilde\Phi_\lambda)$, then $u_-\le u\le u_+$ in $\Omega$ and $u\in\cs$ solves \eqref{dir}.
\end{enumroman}
\end{lemma}
\begin{proof}
We prove \ref{til1}. By ${\bf H}_2$ \ref{h21} we see that $\tilde{f}_\lambda:\Omega\times\R\to\R$ satisfies ${\bf H}_0$, so $\tilde\Phi_\lambda\in C^1(\w)$ with derivative given for all $u,\varphi\in\w$ by
\beq\label{til3}
\langle\tilde\Phi_\lambda'(u),\varphi\rangle = \langle\fpl u,\varphi\rangle-\int_\Omega\tilde{f}_\lambda(x,u)\varphi\,dx.
\eeq
It is easily seen that $\tilde\Phi_\lambda$ is coercive in $\w$. Indeed, since $u_\pm\in\cs$, the mapping $\kappa$ is bounded in $\Omega\times\R$, hence by ${\bf H}_2$ \ref{h21} $\tilde{f}_\lambda$ is bounded as well. So, there exists $C>0$ s.t.\ for a.e.\ $x\in\Omega$ and all $t\in\R$
\[\tilde{F}_\lambda(x,t) \le C|t|.\]
So, for all $u\in\w$ we have
\begin{align*}
\tilde\Phi_\lambda(u) &\ge \frac{\|u\|^p}{p}-\int_\Omega C|u|\,dx \\
&\ge \frac{\|u\|^p}{p}-C\|u\|,
\end{align*}
and the latter tends to $\infty$ as $\|u\|\to\infty$.
\vskip2pt
\noindent
Next we prove that $\tilde\Phi_\lambda$ satisfies $(PS)$. Let $(u_n)$ be a sequence in $\w$ s.t.\ $|\tilde\Phi_\lambda(u_n)|\le C$ for all $n\in\N$, and $\tilde\Phi_\lambda'(u_n)\to 0$ in $W^{-s,p'}(\Omega)$. By coercivity, $(u_n)$ is bounded in $\w$. Passing to a subsequence, we have $u_n\rightharpoonup u$ in $\w$, $u_n\to u$ in $L^1(\Omega)$. By \eqref{til3} we have for all $n\in\N$ and $\varphi\in\w$
\beq\label{til4}
\langle\fpl u_n,\varphi\rangle = \int_\Omega\tilde{f}_\lambda(x,u_n)\varphi\,dx+{\bf o}(1).
\eeq
Testing \eqref{til4} with $\varphi=u_n-u\in\w$ we have
\begin{align*}
\langle\fpl u_n,u_n-u\rangle &= \int_\Omega\tilde{f}_\lambda(x,u_n)(u_n-u)\,dx+{\bf o}(1) \\
&\le C\|u_n-u\|_1+{\bf o}(1),
\end{align*}
and the latter tends to $0$ as $n\to\infty$. By the $(S)_+$-property of $\fpl$, we deduce that $u_n\to u$ in $\w$, so $\tilde\Phi_\lambda$ satisfies $(PS)$.
\vskip2pt
\noindent
Now we prove \ref{til2}. Let $u\in K(\tilde\Phi_\lambda)$. First we see that $u\le u_+$ in $\Omega$. Testing \eqref{til4} with $(u-u_+)^+\in\w$, and recalling the definition of $\tilde{f}_\lambda$, we have
\begin{align*}
\langle\fpl u,(u-u_+)^+\rangle &= \int_\Omega\tilde{f}_\lambda(x,u)(u-u_+)^+\,dx \\
&= \int_\Omega\big[\lambda(u_+)^{q-1}+g(x,u_+)\big](u-u_+)^+\,dx \\
&= \langle\fpl u_+,(u-u_+)^+\rangle.
\end{align*}
Arguing as in the proof of \cite[Lemma 3.2]{FI}, we see that
\[\|(u-u_+)^+\|^p \le C\langle\fpl u-\fpl u_+,(u-u_+)^+\rangle = 0,\]
hence $(u-u_+)^+=0$. Similarly, we prove that $u\ge u_-$ in $\Omega$. Again by the definition of $\tilde{f}_\lambda$, we see that weakly in $\Omega$
\[\fpl u = f_\lambda(x,u),\]
i.e., $u$ is a solution of \eqref{dir}. By Proposition \ref{reg}, we have $u\in\cs$.
\end{proof}

\noindent
By ${\bf H}_2$ \ref{h24}, it is easily seen that $0\in K(\tilde\Phi_\lambda)$. Without loss of generality, we may assume that $0$ is an isolated critical point, i.e., that there exists a neighborhood $U$ of $0$ s.t.\
\[K(\tilde\Phi_\lambda)\cap U=\{0\}.\]
Thus, we may compute the critical groups of $\tilde\Phi_\lambda$ at $0$ (see \cite[Definition 6.43]{MMP}):

\begin{lemma}\label{cgz}
Let ${\bf H}_2$ hold. Then, for all $\lambda>0$, $k\in\N$
\[C_k(\tilde\Phi_\lambda,0) = 0.\]
\end{lemma}
\begin{proof}
Preliminarily we establish some precise estimates on $\tilde{F}_\lambda$. First, by ${\bf H}_2$ \ref{h21} \ref{h24}, for all $\eps>0$ we can find $C_\eps>0$ s.t.\ for a.e.\ $x\in\Omega$ and all $t\in\R$
\[G(x,t) \ge -\eps|t|^p-C_\eps|t|^r.\]
So, for all $u_-(x)\le t\le u_+(x)$ we have
\[G(x,t) \ge -\eps|t|^p-C_\eps\max\{\|u_+\|_\infty,\,\|u_-\|_\infty\}^{r-p}|t|^p \ge -C|t|^p.\]
Then, for any $\lambda>0$ we get
\beq\label{cgz1}
\tilde{F}_\lambda(x,t) \ge \frac{\lambda}{q}|t|^q-C|t|^p.
\eeq
Now, fix $\mu\in(q,p)$. By ${\bf H}_2$ \ref{h21} we have
\begin{align*}
\mu\tilde{F}_\lambda(x,t)-\tilde{f}_\lambda(x,t)t &\ge -\mu\Big[\frac{\lambda}{q}|t|^q+C(|t|+|t|^r)\Big]-\big[\lambda|t|^{q-1}+C(1+|t|^{r-1})\big]|t| \\
&\ge -C_\mu(1+|t|^r),
\end{align*}
with $C_\mu>0$ depending on $\mu$. The latter inequality implies
\beq\label{cgz2}
\liminf_{t\to\infty}\frac{\mu\tilde{F}_\lambda(x,t)-\tilde{f}_\lambda(x,t)t}{t^r} > -\infty,
\eeq
uniformly for a.e.\ $x\in\Omega$. Besides, by ${\bf H}_2$ \ref{h24} \ref{h26} we can find $\delta\in(0,\delta_1]$ s.t.\ for a.e.\ $x\in\Omega$ and all $|t|\le\delta$ we have both
\[|g(x,t)| \le |t|^{p-1}, \ G(x,t) \ge 0.\]
We claim that, by taking $\delta>0$ even smaller if necessary, for a.e.\ $x\in\Omega$ and all $0<|t|\le\delta$ we have
\beq\label{cgz3}
\mu\tilde{F}_\lambda(x,t)-\tilde{f}_\lambda(x,t)t > 0.
\eeq
Indeed, pick $x\in\Omega$, $t\in(0,\delta]$ and distinguish two cases:
\begin{itemize}[leftmargin=1cm]
\item[$(a)$] if $t>u_+(x)$, then
\begin{align*}
\tilde{F}_\lambda(x,t) &= \int_0^{u_+(x)}\big[\lambda\tau^{q-1}+g(x,\tau)\big]\,d\tau+\int_{u_+(x)}^t\big[\lambda u_+(x)^{q-1}+f(x,u_+(x))\big]\,d\tau \\
&= \frac{\lambda}{q}u_+(x)^q+G(x,u_+(x))+\big[\lambda u_+(x)^{q-1}+g(x,u_+(x))\big](t-u_+(x));
\end{align*}
\item[$(b)$] if $0<t\le u_+(x)$, then simply
\[\tilde{F}_\lambda(x,t) = \frac{\lambda}{q}t^q+G(x,t).\]
\end{itemize}
In any case, we have
\begin{align*}
\mu\tilde{F}_\lambda(x,t)-\tilde{f}_\lambda(x,t)t &= \mu\Big[\frac{\lambda}{q}\kappa(x,t)^q+G(x,\kappa(x,t))\Big]+\big[\lambda\kappa(x,t)^{q-1}+g(x,\kappa(x,t))\big]\big[(t-u_+(x))^+-t\big] \\
&\ge \frac{\mu\lambda}{q}\kappa(x,t)^q-\big[\lambda\kappa(x,t)^{q-1}+g(x,\kappa(x,t))\big]\kappa(x,t) \\
&= \Big(\frac{\mu}{q}-1\Big)\lambda\kappa(x,t)^q-g(x,\kappa(x,t))\kappa(x,t) \\
&\ge C_1\kappa(x,t)^q-C_2\kappa(x,t),
\end{align*}
with $C_1,C_2>0$ (recall that $\mu>q$). Here we have used the equality
\[(t-u_+(x))^+-t = -\kappa(x,t),\]
holding for all $t>0$, along with $G(x,t)\ge 0$ and the relations in $(a)$, $(b)$. Since $p>q$ and $\kappa(x,t)\le t$, for all $t>0$ small enough we deduce
\[\mu\tilde{F}_\lambda(x,t)-\tilde{f}_\lambda(x,t)t > 0.\]
Similarly, we deal with $t\in[-\delta,0)$, thus proving \eqref{cgz3}. Combining \eqref{cgz2} and \eqref{cgz3}, we find $C>0$ s.t.\ for a.e.\ $x\in\Omega$ and all $t\in\R$
\beq\label{cgz4}
\mu\tilde{F}_\lambda(x,t)-\tilde{f}_\lambda(x,t)t > -C|t|^r.
\eeq
Armed with the estimates above, we can describe the behavior of $\tilde\Phi_\lambda$ near $0$. First, fix $\rho>0$ s.t.\
\[K(\tilde\Phi_\lambda)\cap\overline{B}_\rho(0) = \{0\}.\]
For any $v\in\w\setminus\{0\}$ s.t.\ $\tilde\Phi_\lambda(v)=0$, the mapping $\tau\mapsto\tilde\Phi_\lambda(\tau v)$ is $C^1$ in $(0,\infty)$ and, by the chain rule, we have
\begin{align*}
\frac{d}{d\tau}\restr{\tilde\Phi_\lambda(\tau v)}{\tau=1} &= \langle\tilde\Phi_\lambda'(v),v\rangle-\mu\tilde\Phi_\lambda(v) \\
&= \Big(1-\frac{\mu}{p}\Big)\|v\|^p+\int_\Omega\big[\mu\tilde{F}_\lambda(x,v)-\tilde{f}_\lambda(x,v)v\big]\,dx \\
&\ge \Big(1-\frac{\mu}{p}\Big)\|v\|^p-C\|v\|^r,
\end{align*}
where  we have used \eqref{cgz4}. Since $\mu<p<r$, the latter is positive whenever $\|v\|>0$ is small enough. So, taking $\rho>0$ even smaller if necessary, for all $v\in\overline{B}_\rho(0)\setminus\{0\}$ s.t.\ $\tilde\Phi_\lambda(v)=0$ we have
\beq\label{cgz5}
\frac{d}{d\tau}\restr{\tilde\Phi_\lambda(\tau v)}{\tau=1} > 0.
\eeq
Now consider $u\in\overline{B}_\rho(0)\cap\cs\setminus\{0\}$. Since $u_\pm\in\pm{\rm int}(\cs_+)$, for all $\tau>0$ small enough we have $u_\pm-\tau u\in\pm{\rm int}(\cs_+)$, in particular $u_-<\tau u<u_+$ in $\Omega$. So, by \eqref{cgz1} we have
\begin{align*}
\tilde\Phi_\lambda(\tau u) &\le \frac{\tau^p}{p}\|u\|^p-\int_\Omega\Big[\frac{\lambda}{q}|\tau u|^q-C|\tau u|^p\Big]\,dx \\
&= \Big[\frac{\|u\|^p}{p}+C\|u\|_p^p\Big]\tau^p-\frac{\lambda}{q}\|u\|_q^q\tau^q,
\end{align*}
and the latter is negative for all $\tau>0$ small enough (depending on $u$). The same holds for all $u\in\overline{B}_\rho(0)\setminus\{0\}$ by density (see \cite[Theorem 6]{FSV}), so we may set
\beq\label{cgz6}
\tau^*(u) = \inf\big\{\tau>0: \ \tilde\Phi_\lambda(\tau u)>0\big\} > 0.
\eeq
Define the closed set
\[D = \big\{u\in\overline{B}_\rho(0): \ \tilde\Phi_\lambda(u)\le 0\big\},\]
which is nonempty due to \eqref{cgz6}. We claim that $D$ is contractible (see \cite[Definition 6.22]{MMP}). First we prove that, for all $u\in D\setminus\{0\}$ and all $\tau\in[0,1]$, we have $\tau u\in D$. Arguing by contradiction, let $u\in D$, $\tau_0\in(0,1)$ s.t.\
\[\tilde\Phi_\lambda(\tau_0 u) > 0.\]
Since $\tilde\Phi_\lambda(u)\le 0$, by the mean value theorem we can find $\tau_1\in(\tau_0,1]$ s.t.\ $\tilde\Phi_\lambda(\tau_1 u)=0$. Set
\[\tau_2 = \min\big\{\tau\in(\tau_0,1]: \ \tilde\Phi_\lambda(\tau u)=0\big\}.\]
Then $\tau_2>\tau_0$ and $\tilde\Phi_\lambda(\tau u)>0$ for all $\tau\in[\tau_0,\tau_2)$, which by monotonicity implies
\[\frac{d}{d\tau}\restr{\tilde\Phi_\lambda(\tau u)}{\tau=\tau_2} \le 0.\]
Besides, by \eqref{cgz5} with $v=\tau_2u\in\overline{B}_\rho(0)\setminus\{0\}$ and the chain rule we have
\[\frac{d}{d\tau}\restr{\tilde\Phi_\lambda(\tau u)}{\tau=\tau_2} = \frac{1}{\tau_2}\frac{d}{d\tau}\restr{\tilde\Phi_\lambda(\tau v)}{\tau=1} >0,\]
a contradiction. So $D$ is star-shaped, hence contractible by \cite[Remark 6.23]{MMP}. Now set
\[D_0 = \big\{u\in\overline{B}_\rho(0)\setminus\{0\}:\,\tilde\Phi_\lambda(u)\le 0\big\}, \ E_0 = \big\{u\in\overline{B}_\rho(0)\setminus\{0\}:\,\tilde\Phi_\lambda(u)> 0\big\},\]
so that $D_0\cup E_0=\overline{B}_\rho(0)\setminus\{0\}$. We prove now that $D_0$ is contractible. Indeed, for all $u\in E_0$, by \eqref{cgz6} there exists $\tau(u)\in(0,1)$ s.t.\
\[\tilde\Phi_\lambda(\tau(u)u) = 0.\]
By \eqref{cgz5} and the implicit function theorem, $\tau(u)\in(0,1)$ is unique and the map $\tau:E_0\to(0,1)$ is continuous. So, set for all $u\in\overline{B}_\rho(0)\setminus\{0\}$
\[j(u) = \begin{cases}
u & \text{if $u\in D_0$} \\
\tau(u)u & \text{if $u\in E_0.$}
\end{cases}\]
The map $j:(\overline{B}_\rho(0)\setminus\{0\})\to D_0$ is continuous. Indeed, avoiding trivial cases, let $(u_n)$ be a sequence in $E_0$ s.t.\ $u_n\to u$ in $\w$, for some $u\in D_0$. Then we have $\tilde\Phi_\lambda(u)=0$, hence by uniqueness $\tau(u_n)\to 1$, which in turn implies
\[\lim_n j(u_n) = u =j(u).\]
Recalling that $j(u)=u$ for all $u\in D_0$, we conclude that $j$ is a retraction of $\overline{B}_\rho(0)\setminus\{0\}$ onto $D_0$. Since $\w$ is infinite-dimensional, then $\overline{B}_\rho(0)\setminus\{0\}$ is contractible, hence $D_0$ is contractible as well. Finally, by the excision property of critical groups and \cite[Propositions 6.24, 6.25]{MMP}, we have for all $k\in\N$
\[C_k(\tilde\Phi_\lambda,0) = H_k(D,D_0) = H_k(D,\star) = 0,\]
which proves the assertion.
\end{proof}

\noindent
We can finally prove our multiplicity result:

\begin{theorem}\label{nod}
Let ${\bf H}_2$ hold. Then, there exists $\tilde\lambda>0$ s.t.\ for all $\lambda\in(0,\tilde\lambda)$ problem \eqref{dir} has at least four nontrivial solutions: $u_+,v_+\in{\rm int}(\cs_+)$, $u_-\in -{\rm int}(\cs_+)$, and $\tilde u\in\cs$ nodal.
\end{theorem}
\begin{proof}
Once again we remark that hypotheses ${\bf H}_2$ imply ${\bf H}_1$, so let $\lambda^*,\lambda_*>0$ be defined by Theorem \ref{css} and Lemma \ref{ext}, respectively, and $\tilde\lambda>0$ by \eqref{tla}. As above, we assume that $u_\pm\in\pm{\rm int}(\cs_+)$ are the extremal constant sign solutions of \eqref{dir} and $v_+\ge u_+$ in $\Omega$, and accordingly define $\tilde\Phi_\lambda\in C^1(\w)$. Finally, without loss of generality we assume that $K(\tilde\Phi_\lambda)$ is a finite set.
\vskip2pt
\noindent
First we prove that $u_+$ is a local minimizer of $\tilde\Phi_\lambda$. Indeed, set for all $(x,t)\in\Omega\times\R$
\[\tilde{f}_\lambda^+(x,t) = \tilde{f}_\lambda(x,t^+), \ \tilde{F}_\lambda^+(x,t) = \int_0^t\tilde{f}_\lambda^+(x,\tau)\,d\tau,\]
and for all $u\in\w$
\[\tilde\Phi_\lambda^+(u) = \frac{\|u\|^p}{p}-\int_\Omega\tilde{F}_\lambda^+(x,u)\,dx.\]
Arguing as in Lemma \ref{til} we see that $\tilde\Phi_\lambda^+\in C^1(\w)$ is coercive, satisfies $(PS)$, and whenever $u\in K(\tilde\Phi_\lambda^+)$ we have that $u\in\cs$ solves \eqref{dir} and $0\le u\le u_+$ in $\Omega$. So, there exists $\tilde{u}_+\in\w$ s.t.\
\[\tilde\Phi_\lambda^+(\tilde{u}_+) = \inf_{u\in\w}\tilde\Phi_\lambda^+(u) = \tilde{m}_+.\]
Using ${\bf H}_2$ \ref{h24} as in Lemma \ref{pos} (precisely, see \eqref{pos3}) we see that $\tilde{m}_+<0$, hence $\tilde{u}_+\neq 0$. Once again, Propositions \ref{reg}, \ref{smp} imply that $\tilde{u}_+\in{\rm int}(\cs_+)$. So, $\tilde{u}_+$ turns out to be a positive solution of \eqref{dir} s.t.\ $\tilde{u}_+\le u_+$ in $\Omega$, which by extremality implies $\tilde{u}_+=u_+$. Then, for all $u\in\w\cap{\rm int}(\cs_+)$ we have
\[\tilde\Phi_\lambda(u) = \tilde\Phi_\lambda^+(u) \ge \tilde\Phi_\lambda^+(u_+) = \tilde\Phi_\lambda(u_+),\]
in particular $u_+$ is a $\cs$-local minimizer of $\tilde\Phi_\lambda$. By Proposition \ref{svh}, $u_+$ is also a $\w$-local minimizer of $\tilde\Phi_\lambda$, as claimed.
\vskip2pt
\noindent
Similarly, we see that $u_-\in -{\rm int}(\cs_+)$ is a local minimizer of $\tilde\Phi_\lambda$.
\vskip2pt
\noindent
Recalling that $K(\tilde\Phi_\lambda)$ is finite, by a topological version of the mountain pass theorem (see \cite[Theorem 6.99, Proposition 6.100]{MMP}) we deduce the existence of $\tilde{u}\in K(\tilde\Phi_\lambda)$ s.t.\ $\tilde{u}\neq u_\pm$ and
\beq\label{nod1}
C_1(\tilde\Phi_\lambda,\tilde{u}) \neq 0.
\eeq
Comparing \eqref{nod1} with Lemma \ref{cgz}, we see that $\tilde{u}\neq 0$. Besides, by Lemma \ref{til} $\tilde{u}\in\cs\setminus\{0\}$ solves \eqref{dir} and $u_-\le\tilde{u}\le u_+$ in $\Omega$. Then, $\tilde{u}$ must change sign in $\Omega$. Indeed, assuming by contradiction that $\tilde{u}\ge 0$, then by Proposition \ref{smp} we would have $\tilde{u}\in{\rm int}(\cs_+)$ with $\tilde{u}\le u_+$ and $\tilde{u}\not\equiv u_+$, a contradiction to Lemma \ref{ext}. Similarly, if $\tilde{u}\le 0$ in $\Omega$, we reach a contradiction.
\vskip2pt
\noindent
Thus, we have proved the existence of four solutions of \eqref{dir} (beside $0$): $u_+,v_+\in{\rm int}(\cs_+)$, $u_-\in -{\rm int}(\cs_+)$, and $\tilde{u}\in\cs$ nodal.
\end{proof}

\begin{remark}\label{rno}
Again we recall some alternative assumptions to ${\bf H}_2$, under which existence of a nodal solution can be achieved. For instance, arguing as in \cite{MMP1} one could require a linear behavior of $g(x,t)$ as $t\to 0^-$, together with a global sign condition. As in \cite{IMP}, one could assume a quasi-critical growth with a quasi-monotonicity condition on $f_\lambda(x,\cdot)$ (see Remark \ref{rcs}). Finally, as in \cite{FI}, one can assume a different condition of the type
\[\liminf_{t\to 0}\frac{f_\lambda(x,t)}{|t|^{p-2}t} \ge \lambda_2\]
uniformly for a.e.\ $x\in\Omega$, where $\lambda_2>\lambda_1$ denotes the second variational eigenvalue of $\fpl$ in $\w$ (this argument is based on a variational characterization of $\lambda_2$ proved in \cite{BP}).
\end{remark}

\vskip4pt
\noindent
{\bf Acknowledgement.} Both authors are members of GNAMPA (Gruppo Nazionale per l'Analisi Matematica, la Probabilit\`a e le loro Applicazioni) of INdAM (Istituto Nazionale di Alta Matematica 'Francesco Severi') and are supported by the grant PRIN n.\ 2017AYM8XW: {\em Non-linear Differential Problems via Variational, Topological and Set-valued Methods}. A.\ Iannizzotto is also supported by the research project {\em Evolutive and stationary Partial Differential Equations with a Focus on Biomathematics} (Fondazione di Sardegna 2019). We wish to thank D.\ Mugnai for a stimulating discussion on the Brezis-Oswald result. Finally, we thank the anonymous Referee for the useful suggestions.

\end{document}